\documentclass[10pt,a4paper]{article}

\usepackage{amsfonts}
\usepackage{amssymb}
\usepackage{amsmath}
\usepackage{amsthm}
\usepackage[english]{babel}
\usepackage{enumerate}
\usepackage{graphicx}

\newcommand{\NN}{\mathbb{N}}
\newcommand{\RR}{\mathbb{R}}

\newcommand{\ints}{\int\limits}

\newcommand{\lam}{\lambda}

\newcommand{\ep}{\varepsilon}

\newcommand{\de}{\delta}

\newcommand{\al}{\alpha}
\newcommand{\si}{\sigma}
\newcommand{\ga}{\gamma}

\newcommand{\RRlxi}{\mathcal{R}_\lam}

\newcommand{\II}{\mathcal{I}}

\newcommand{\Om}{\Omega}
\newcommand{\pa}{\partial}
\newcommand{\RN}{\RR^N}
\newcommand{\ovr}{\overline}
\newcommand{\Ga}{\Gamma}

\newcommand{\De}{\Delta}

\newtheorem{theorem}{Theorem}[section]

\newtheorem{lemma}[theorem]{Lemma}
\newtheorem{proposition}[theorem]{Proposition}

\theoremstyle{remark}
\newtheorem{rmk}[theorem]{Remark}

\DeclareMathOperator{\diver}{\mathrm{div}}

\DeclareMathOperator{\dist}{\mathrm{dist}}

\renewenvironment{proof}{
  \noindent{\it Proof.}\ }{\hspace*{\fill}
  \begin{math}\Box\end{math}\medskip}

\title{Symmetry of minimizers with a level surface
parallel to the boundary\thanks{This research was partially supported by Grants-in-Aid
for Scientific Research (B) ($\sharp$ 20340031) of
Japan Society for the Promotion of Science and by a
Grant of the Ital\-ian MURST. The first two authors have been also supported by the Gruppo Nazionale per l'Analisi Matematica, la Probabilit\`a e le loro Applicazioni (GNAMPA) of the Italian Istituto Nazionale di Alta Matematica (INdAM).
}}

\author{Giulio Ciraolo\thanks{Dipartimento di Matematica e
Informatica, Universit\`a di Palermo, Via Archirafi 34, 90123, Italy, ({\tt giulio.ciraolo@unipa.it}).}
\\ Rolando Magnanini\thanks{Dipartimento di matematica ``U.
Dini'', Universit\`a di Firenze, Viale Morgagni 67/A, 50134 Firenze,
Italy, ({\tt magnanin@math.unifi.it}).}
\\ Shigeru Sakaguchi\thanks{Department of Applied Mathematics,
Graduate School of  Engineering, Hiroshima
University, Higashi-Hiroshima, 739-8527,  Japan.
({\tt sakaguch@amath.hiroshima-u.ac.jp}).}}
\begin{document}

\date{}

\maketitle

%%% \tableofcontents

\begin{abstract}
We consider the functional
\begin{equation*}
\II_\Om(v) = \int_\Omega [f(|Dv|) - v]\,dx,
\end{equation*}
where $\Omega$ is a bounded domain and $f$ is a convex function. Under general assumptions
on $f$, Crasta \cite{Cr1} has shown that if $\II_\Om$ admits a minimizer in $W_0^{1,1}(\Omega)$ depending only on the distance from the boundary of $\Omega$, then $\Omega$ must be a ball. With some restrictions on $f$, we prove that spherical symmetry can be obtained only by assuming that the minimizer has {\it one} level surface parallel to the boundary
(i.e. it has only a level surface in common with the distance).

We then discuss how these results extend to more general settings, in particular to functionals that are not differentiable and to solutions of fully nonlinear elliptic and parabolic equations.
\end{abstract}

\section{Introduction} \label{section introd}

We consider a bounded domain $\Om$ in $\RR^N$ $(N\geq 2)$ and,
for $x\in\ovr{\Om},$ denote by $d(x)$ the distance of $x$ from $\RN\setminus\Om,$ that is
$$
d(x)=\min_{y\in\RR^N\setminus\Om}|x-y|, \ x\in\ovr{\Om};
$$
$d$ is Lipschitz continuous on $\ovr{\Om}.$  For a positive number $\de,$ we define
the {\it parallel surface} to the boundary $\pa\Om$ of $\Om$ as
$$
\Ga_\de=\{ x\in\Om: d(x)=\de\}.
$$
\par
In this paper, we shall be concerned with minimizers of variational problems and solutions of quite general
nonlinear elliptic and parabolic partial differential equations,
which admit a single level surface that is parallel to $\pa\Om.$
\par
A motivation to our concern is the work of G. Crasta \cite{Cr1} on
the minimizers of certain problems of the Calculus of Variations in the class of
the so-called {\it web-functions}, that is those functions that
depend only on the distance from the boundary (see \cite{CG} and \cite{Ga}, where the term
web-function was introduced for the first time).
In \cite{Cr1}, it is proved that, if $\Om$ is a smooth domain and the functional
\begin{equation}
\label{defJ}
\II_\Om(v)=\int_\Om [ f(|D v|) -v ]\, dx
\end{equation}
has a minimizer in the class of $W^{1,1}_0(\Om)$-regular web functions, then
$\Om$ must be a ball. The assumptions on the lagrangean $f$ are very general:
$f$ is merely required to be convex and the function $p\mapsto f(|p|)$ to be differentiable.
\par
Related to Crasta's result, here, we consider the variational problem
\begin{equation}\label{I(u)}
\inf \{\II_\Om(v):\ v \in W_0^{1,\infty}(\Omega)\},
\end{equation}
under the following assumptions for $f:[0,\infty)\to\RR:$
\begin{itemize}
\item[(f1)] $f\in C^1([0,+\infty))$ is a convex, monotone nondecreasing function such
that $f(0)=0$ and
$$\lim\limits_{s\to +\infty} \dfrac{f(s)}{s} = +\infty;$$

\item[(f2)] there exists $\si \geq 0$ such that $f'(s)=0$ for every $0 \leq s \leq \si$, $f'(s)>0$ for $s>\si$ and $f\in C^{2,\al}(\si,+\infty)$ $(0<\al<1),$ with $f''(s)>0$ for $s>\si.$
\end{itemize}
Also, we suppose that there exists a domain $G$ such that
\begin{equation}
\label{defD}
\ovr{G}\subset\Om,\, \pa G\in C^1 \mbox{ satisfying the interior sphere condition, and } \pa G=\Ga_\de,
\end{equation}
for some $\de>0.$
\par
The main result in this paper is the following.

\begin{theorem} \label{thm symmetry diff functionals}
Let $\Om\subset\RN$ be a bounded domain and let
$f$ and $G$ satisfy assumptions (f1)-(f2) and \eqref{defD}, respectively.
\par
Let $u$ be the solution of \eqref{I(u)} and suppose $u$ is $C^1$-smooth in a tubular neighborhood of $\Ga_\de.$
\par
If
\begin{equation}
\label{OVDTcond}
u=c \ \mbox{ on } \ \Ga_\de
\end{equation}
for some constant $c>0,$
then $\Om$ must be a ball.
\end{theorem}

Thus, at the cost of requiring more restrictive growth and regularity assumptions
on $f,$ we can sensibly improve Crasta's theorem:
indeed, if $u$ is a web function, then {\it all} its level surfaces are parallel to $\pa\Om.$
We also point out that we make no (explicit) assumption on the regularity of $\pa\Om:$
we only require that the parallel surface $\Ga_\de$ has some special topology and
is mildly smooth.
\par
In Theorem \ref{thm symmetry not-diff functionals}, we will also extend this result to a case
in which the function $p\mapsto f(|p|)$ is no longer differentiable at $p=0.$
Our interest on this kind of functionals (not considered in \cite{Cr1}) is
motivated by their relevance in the study of complex-valued solutions of the {\it eikonal} equation (see \cite{MT1}-\cite{MT4} and \cite{CeM}).
\par

The reason why we need stricter assumptions on $f$ rests upon a different method of proof.
While in \cite{Cr1} one obtains symmetry by directly working on the Euler's equation for $\II_\Om$
(that only involves $f'$), in the proof of Theorem \ref{thm symmetry diff functionals}, we rely on
the fact that minimizers of $\II_\Om$ satisfy in a generalized sense a nonlinear equation of type
\begin{equation}
\label{elliptic}
F(u, D u, D^2 u)=0 \ \mbox{ in } \ \Om,
\end{equation}
(that involves $f''$); moreover, to obtain symmetry, we use the {\it method of moving planes}
that requires extra regularity for $f''.$
\par
Theorem \ref{thm symmetry diff functionals} (and also Theorems \ref{thm symmetry not-diff functionals}, \ref{th:symmEll} and \ref{th:symmPar}) works out an idea used by
the last two authors in the study of the so-called {\it stationary
surfaces} of solutions of (non-degenerate) fast-diffusion parabolic
equations (see \cite{MS}).  A stationary surface is a surface $\Ga\subset\Om$ of codimension $1$
such that, for some function $a:(0,T)\to\RR,$ $u(x,t)=a(t)$ for every $(x,t)\in \Ga\times(0,T).$
In fact, in \cite{MS}, it is proved that if the
initial-boundary value problem
\begin{eqnarray*}
\label{fastdiffusion}
&u_t-\De\phi(u)=0 \ \mbox{ in } \ \Om\times(0,T),\\
&u=0\ \mbox{ on } \ \Om\times\{0\},\ \ u=1 \ \mbox{ on } \ \pa\Om\times(0,T)
\end{eqnarray*}
(here $\phi$ is a nonlinearity with derivative $\phi'$ bounded
from below and above by positive constants), admits a solution
that has a stationary surface, then $\Om$ must be a ball.
\par
The crucial arguments used in \cite{MS} are two: the former is the
discovery that a stationary surface must be parallel to
the boundary; the latter is the application of the method of moving planes.
This method was created by A.V. Aleksandrov to prove the spherical
symmetry of embedded surfaces with constant mean curvature or, more generally,
whose principal curvatures satisfy certain
constraints and, ever since, it has been successfully employed to prove spherical symmetry in
many a situation: the theorems of Serrin's for overdetermined boundary value problems (\cite{Se2}) and
those of Gidas, Ni and Nirenberg's for ground states (\cite{GNN}) are the most celebrated.
Here, we will use that method to prove Theorem \ref{thm symmetry diff functionals} (and also Theorems \ref{thm symmetry not-diff functionals}, \ref{th:symmEll} and \ref{th:symmPar}). Arguments similar to those used in \cite{MS} were
recently used in \cite{Sh}.

Let us now comment on the connections between the problem considered in Theorem \ref{thm symmetry diff functionals},
the one studied in \cite{Cr1} (both with $f(p)=\frac12\,|p|^2$)
and (the simplest instance of) Serrin's overdetermined
problem:
\begin{eqnarray}
&-\De u=1 \ \mbox{ in } \ \Om, \label{serrin1}\\
&\displaystyle u=0 \mbox{ on } \ \pa\Om,\label{Dirichlet}  \ \
\dfrac{\pa u}{\pa \nu}=\mbox{\rm constant} \mbox{ on } \ \pa\Om. \label{serrin2}
\end{eqnarray}
\par
It is clear that being a web function is a stronger condition,
since it implies both \eqref{OVDTcond} and \eqref{serrin2}.
Moreover, even if the constraint \eqref{OVDTcond} and
\eqref{serrin2} are not implied by one another,
we observe the following: (i) if \eqref{OVDTcond}
is verified for two positive sequences
$\{\de_n\}_{n\in\NN}$ and $\{c_n\}_{n\in\NN}$ with $\de_n\to 0$ as $n\to\infty,$ then
\eqref{serrin2} holds true;
(ii) from \eqref{Dirichlet} instead we can conclude that the
oscillation $\max\limits_{\Ga_\de} u-\min\limits_{\Ga_\de} u$ is $O(\de^2)$ as $\de\to 0.$
All in all, it seems that the constraint \eqref{OVDTcond} is weaker than \eqref{serrin2}.
\par
Another important remark is in order: the method of moving planes is
applied to prove our symmetry results in a much simplified
manner than that used for \eqref{serrin1}-\eqref{serrin2}; indeed,
since the overdetermination takes place in $\Om$ (and not on $\pa\Om$)
we need not use Serrin's corner lemma (in other words,
property (B) in \cite{Se2} for \eqref{elliptic} is not required). A further
benefit of this fact is that no regularity requirement is made on $\pa\Om,$
thanks to assumption \eqref{defD}.
\par

In Section \ref{section minima diff}
we will present our results on the problem proposed by Crasta
(for the proof of Theorem \ref{thm symmetry diff functionals}, see Subsection \ref{subsection proof of thm}.); in Section \ref{section minima not diff}
we will extend them to some cases which involve non-differentiable lagrangeans.
In Section \ref{section4} we will discuss how these results extend to fairly general settings,
in particular to solutions of fully nonlinear elliptic and parabolic equations.

We mention that a stability version of Theorem \ref{thm symmetry diff functionals} (for the semilinear
equation $\De u=f(u)$) is obtained in the companion paper \cite{CMS}.

\section{Minima of convex differentiable functionals}\label{section minima diff}

We first introduce some notation and prove some preliminary result.

\subsection{Uniqueness and comparison results}\label{subsection comp results}
Let $f$ satisfy (f1)-(f2). The functional $\II_\Om$ is
differentiable and a critical point $u$ of $\II_\Om$ satisfies the
problem
\begin{equation} \label{Euler eq F2I 1}
\begin{cases}
- \diver \left( \dfrac{f'(|Du|)}{|Du|} Du \right) = 1, & \textmd{in } \Omega,\\
u=0, &  \textmd{on } \pa \Omega,
\end{cases}
\end{equation}
in the weak sense, i.e.
\begin{equation}\label{Euler eq F2I 2}
\ints_\Omega \dfrac{f'(|Du|)}{|Du|} Du \cdot D\phi dx = \ints_\Omega \phi dx, \quad \textmd{for every } \phi\in C_0^1(\Omega).
\end{equation}

It will be useful in the sequel to have at hand the solution of \eqref{Euler eq F2I 1} when
$\Om$ is the ball of given radius $R$ (centered at the origin): it is given by
\begin{equation} \label{eq soluzione palla F2I}
u_R(x) = \ints_{|x|}^R g'\Big( \frac{s}{N} \Big) ds ,
\end{equation}
where
\begin{equation*}
g(t)=\sup\{st-f(s):\ s\geq 0\}
\end{equation*}
is the Fenchel conjugate of $f$. For future use, we notice that $|Du_R(x)| > \si$ for $x\not=0$.

It is clear that, when $\si=0,$ \eqref{I(u)} has a unique solution, since $f$ is strictly convex.
When $\si>0,$ proving the uniqueness for \eqref{I(u)} needs some more work.
In Theorem \ref{thm uniq Dir} we shall prove such result as a consequence of Lemmas \ref{lemma sign u} and \ref{lemma uniqueness I} below.

\begin{lemma} \label{lemma sign u}
Let $\Omega$ be a bounded domain and let $u$ be a solution of \eqref{I(u)}, where $f$ satisfies (f1) and (f2), with $\si > 0$.
\par
Then $u\geq 0$ and $\II_\Om(u)<0$.
\end{lemma}
\begin{proof}
Since $u \in W^{1,\infty}_0(\Omega)$, also $|u| \in W^{1,\infty}_0(\Omega).$ If $u<0$ on some open subset of $\Om$ ($u$ is continuous), then $\II_\Om(|u|) < \II_\Om(u)$ --- a contradiction.
\par
Now observe that $\II_\Om(v) < 0$ if $v \in W^{1,\infty}_0(\Omega)$ is any nonnegative function, $v \not\equiv 0$, with Lipschitz constant less or equal than $\si$. Thus, $\II_\Om(u)<0.$
\end{proof}

In the following, for a given domain $A$, we shall denote by $\II_A$ the integral functional
\begin{equation*}
\II_A(v) = \ints_A [f(|Dv|)-v]dx;
\end{equation*}
a {\it local minimizer} of $\II_A$ means a function that minimizes $\II_A$
among all the functions with the same boundary values.

\begin{lemma} \label{lemma uniqueness I}
Let $f$ satisfy (f1) and (f2) with $\si>0.$ Let $A$ be a bounded domain and
assume that $u_0,u_1 \in W^{1,\infty}(A)$ are local minimizers of $\II_A$, with
$u_0=u_1$ on $\pa A.$
Next, define
\begin{equation} \label{Ej}
E_j = \{x\in A:\ |Du_j|>\si \} ,\quad j=0,1,
\end{equation}
and assume that
\begin{equation*}
|E_0 \cup E_1| > 0.
\end{equation*}
\par
Then $u_0 \equiv u_1$.
\end{lemma}

\begin{proof}
Let $u=\frac12\,(u_0 + u_1);$ since $f$ is convex, it is clear that $u$ is also a minimizer of $\II_A$
and $u=u_0=u_1$ on $\pa A.$
Thus, we have
\begin{equation} \label{eq intA 0}
\int_A \Bigl[ \frac12\, f(|Du_0|) + \frac12\,f(|Du_1|) - f(|D u|) \Bigr] dx = 0,
\end{equation}
and, since $f$ is convex,
\begin{equation} \label{eq intA 0 ae}
\frac12\,  f(|Du_0|) + \frac12\,  f(|Du_1|) - f(|D u|) = 0, \quad \mbox{a.e. in }\ A.
\end{equation}
Assumption (f2) on $f$ and \eqref{eq intA 0 ae} imply that
\begin{equation}\label{eq 30}
|(E_0 \cup E_1) \cap \{|Du_0| \neq |Du_1| \} | = 0,
\end{equation}
since on $(E_0 \cup E_1) \cap \{|Du_0| \neq |Du_1| \}$ the convexity of $f$ holds in the strict sense.
\par
Thus, we have proven that $|Du_0| = |Du_1|$  a.e. in $E_0 \cup E_1$ and, since
\begin{equation*}
\ints_A f(|Du_j|) dx = \ints_{E_j} f(|Du_j|) dx , \quad  j=0,1,
\end{equation*}
we have
\begin{equation*}
\ints_A f(|Du_0|) dx =\ints_A f(|Du_1|) dx .
\end{equation*}
\par
Now, take $v=\max (u_0,u_1)$; \eqref{eq 30} implies that
\begin{equation} \label{eq intA fu0 fu1}
\ints_A f(|Dv|) dx = \ints_A f(|Du_0|) dx =\ints_A f(|Du_1|) dx ,
\end{equation}
and hence
\begin{equation*}
\II_A(v) \leq \II_A(u_0)=\II_A(u_1),
\end{equation*}
since $v \geq u_0,u_1$. Thus, $\II_A(v) = \II_A(u_0)=\II_A(u_1)$; consequently
\begin{equation*}
\ints_A (v-u_j) dx = 0
\end{equation*}
for $j=0, 1$ and, since $v\geq u_0, u_1$, we have that $v=u_0=u_1.$
\end{proof}

\begin{theorem} \label{thm uniq Dir}
Let $f$ satisfy (f1) and (f2) with $\si > 0$ and assume that $u \in W_0^{1,\infty} (\Omega)$ is a solution of \eqref{I(u)}.
\par
Then we have:
\begin{enumerate}[(i)]
\item $|\{x\in \Omega:\ |Du|>\si \}|>0$;
\item $u$ is unique.
\end{enumerate}
\end{theorem}

\begin{proof}
(i) By contradiction, assume that $|Du| \leq \si$ a.e.; since $u$ satisfies \eqref{Euler eq F2I 2}, we can easily infer that
\begin{equation*}
\ints_\Omega u dx=0.
\end{equation*}
Thus, $I(u)=0$, which contradicts Lemma \ref{lemma sign u}.

(ii) The assertion follows from (i) and Lemma \ref{lemma uniqueness I}.
\end{proof}

As already mentioned in the introduction, our proof of Theorem
\ref{thm symmetry diff functionals} makes use of the method of the
moving planes. To apply this method, we need comparison
principles for minimizers of $\II_\Om$.

\begin{proposition}[Weak comparison principle] \label{thm weak comp minimiz}
Let $A$ be a bounded domain and let $f$ satisfy (f1) and (f2).
\par
Assume that $u_0, u_1 \in W^{1,\infty}(A)$ are local minimizers of $\II_A$ such that $u_0 \leq u_1$ on $\pa A$ and suppose that
\begin{equation*}
|E_0 \cup E_1| > 0,
\end{equation*}
where the sets $E_1, E_2$ are given by \eqref{Ej}.
\par
Then $u_0 \leq u_1$ in $\overline{A}$.
\end{proposition}

\begin{proof}
If $\si=0$ in (f2), then the weak comparison principle is well established
(see for instance Lemma 3.7 in [FGK]).

Thus, in the rest of the proof, we assume that $\si > 0$. Assume
by contradiction that $u_0>u_1$ in a non-empty open subset $B$ of $A.$
We can suppose that $B$ is connected (otherwise the
argument can be repeated for each connected
component of $B$). Observe that,
since $u_0 \leq u_1$ on $\pa A$ and $u_0$ and
$u_1$ are continuous, then $u_0=u_1$ on $\partial B$.

We now show that $u_0$ minimizes $\II_B$ among those
functions $v$ such that $\ v - u_0\in W_0^{1,\infty}(B)$. Indeed, if $\inf \II_B(v)
< \II_B(u_0)$ for one such function, then the function $w$
defined by
\begin{equation*}
w(x)= \begin{cases} v, & x \in B, \\
u_0, & x \in \overline{A} \setminus B,
\end{cases}
\end{equation*}
would belong to $W^{1,\infty}(A),$ be equal to $u_0$ on $\pa A$ and be such that
$\II_A(w) < \II_A(u_0)$ --- a contradiction.
The same argument can be repeated for $u_1,$  and hence we have proven that
$\II_B(u_0) = \II_B(u_1)$ (since $u_0=u_1$ on $\pa A$).

This last equality implies that
\begin{equation*}
\ints_B f(|D u_0|) dx > \ints_B f(|D u_1|) dx \geq 0,
\end{equation*}
since $u_0>u_1$ in $B,$ and hence $|E_0 \cap B| > 0$.

By applying Lemma \ref{lemma uniqueness I} to the functional
$\II_B$, we obtain that $u_0 \equiv u_1$ in $B,$ that gives a contradiction.
\end{proof}

\begin{proposition}[Strong comparison principle] \label{lemma F2I strong comparison}
Let $A$ be a bounded domain and let $f$ satisfy (f1) and (f2).
\par
Assume that $u_0,u_1 \in C^1(\overline{A})$ are local minimizers of $\II_A$ such that
$u_0 \leq u_1$ in $A$ and $|Du_0|, |Du_1|>\si$ in $\overline{A}.$
\par
Then either $u_0\equiv u_1$ or else $u_0<u_1$ in $A$.
\end{proposition}

\begin{proof}
Being $u_0$ and $u_1$ solutions of \eqref{Euler eq F2I 2} with $|Du_0|, |Du_1|>\si$, the assertion is easily implied by Theorem 1 in \cite{Se1}, since, by (f2),
the needed uniform ellipticity is easily verified.
\end{proof}

\begin{proposition}[Hopf comparison principle] \label{lemma F2I boundary point}
Let $u_0,u_1 \in C^2(\overline{A})$ be functions satisfying the assumptions of
Proposition \ref{lemma F2I strong comparison}.
Assume that $u_0=u_1$ at some point $P$ on the boundary of $A$ admitting an internally touching tangent sphere.
\par
Then, either $u_0 \equiv u_1$ in $A$ or else
\begin{equation*}
u_0 < u_1 \textmd{ in } A \quad \textmd{ and } \quad \dfrac{\pa u_0}{\pa \nu} < \dfrac{\pa u_1}{\pa \nu} \textmd{ at } P;
\end{equation*}
here, $\nu$ denotes the inward unit normal to $\pa A$ at $P$.
\end{proposition}

\begin{proof}
The assertion follows from Theorem 2 in \cite{Se1} by using an argument analogous to the one in the proof of Proposition \ref{lemma F2I strong comparison}.
\end{proof}

We conclude this subsection by giving a lower bound for $|Du|$ on $\Gamma_\de$
when $u$ is the function considered in Theorem \ref{thm symmetry diff functionals}.

\begin{lemma} \label{lemma F2I gradient boundary}
Let $\Omega,\, G$ and $f$ satisfy the assumptions of Theorem \ref{thm symmetry diff functionals}.

Let $u\in W_0^{1,\infty}(\Omega)$ be a minimizer of \eqref{I(u)}
satisfying \eqref{OVDTcond}. For $x_0\in \Gamma_\de,$ let $y_0\in \pa
\Omega$ be such that $\dist(y_0,\Gamma_\de)=\de$ and set $\nu = \frac{x_0-y_0}{\de}$; denote by $\rho=\rho(x_0)$ the radius of the optimal interior ball at $x_0$.
\par
Then,
\begin{equation}\label{eq F2I gradient bound}
\liminf_{t\to 0^+} \frac{u(x_0+t \nu) - u(x_0)}{t} \geq g'\Big(
\frac{\rho}{N} \Big),
\end{equation}
where $g$ is the Fenchel conjugate of $f$.

In particular, we have that $\ \inf\limits_{\Gamma_\delta} |Du| > \sigma$ in two cases:
\begin{itemize}
\item[(i)] if $u\in C^1(\Gamma_\delta)$;
\item[(ii)] if $u$ is differentiable at every $x\in\Gamma_\de$ and $G$ satisfies the uniform interior sphere condition.
\end{itemize}
\end{lemma}

\begin{proof}
Since $u-c$ minimizes the functional $\II_{G}$ among the functions vanishing on $\Gamma_\de,$
then Lemma \ref{lemma sign u} implies that $u \geq c$ in $\overline{G}$.

Let $B_\rho \subset G$ be the ball of radius $\rho$ tangent to $\Gamma_\de$ in
$x_0$. The minimizer $w$ of $\II_{B_\rho}$ with $w=c$ on $\pa B_\rho$
is then $w=c + u_\rho,$ where $u_\rho$ given by \eqref{eq soluzione palla F2I} with $R=\rho$; notice that $u(x_0)=w(x_0)=c$.
\par
Since $u\geq c \equiv w$ on $\pa B_\rho$, Proposition \ref{thm weak comp minimiz}
yields $u \geq w$ in $B_{\rho}$ and thus
\begin{equation*}
\liminf_{t\to 0^+} \frac{u(x_0 + t\nu) -c}{t} \geq \frac{\pa w}{\pa
\nu}(x_0)=g'\Big( \frac{\rho}{N} \Big).
\end{equation*}
The last part of the lemma (assertions (i) and (ii)) is a straightforward consequence of \eqref{defD} and \eqref{eq F2I gradient bound}.
\end{proof}

\subsection{The proof of Theorem \ref{thm symmetry diff functionals}} \label{subsection proof of thm}
We initially proceed as in \cite{Se2} (see also \cite{Fr}) and further introduce
the necessary modifications as done in \cite{MS}-\cite{MS2}. For $\xi \in\RR^N$ with $|\xi|=1$ and $\lam\in\RR,$ we denote by $\RRlxi x$ the reflection
$x+2(\lam - x\cdot\xi )\,\xi$ of any point $x\in\RR^N$ in the hyperplane
\begin{equation*}
\pi_\lam= \{x\in\RR^N:\ x\cdot \xi = \lam \},
\end{equation*}
and set
\begin{equation*}
u^\lam(x)=u(\RRlxi x) \ \mbox{ for } \ x\in\RRlxi(\Om).
\end{equation*}
Then, for a fixed direction $\xi$, we define the \emph{caps}
$$
G_\lam=\{ x\in G: x\cdot\xi>\lam\} \ \mbox{ and } \ \Om_\lam=\{ x\in\Om: x\cdot\xi>\lam\},
$$
and set
\begin{eqnarray*}
&&\bar{\lam}=\inf\{\lam\in\RR: G_\lam=\varnothing \} \ \mbox{ and }\\
&&\lam^*=\inf\{\lam\in\RR: \mathcal{R}_\mu(G_\mu) \subset G
\ \mbox{ for every } \ \mu\in(\lam,\bar{\lam})\}.
\end{eqnarray*}
\par
As is well-known from Serrin \cite{Se2}, if we assume that $\Om$ (and hence $G$) is not $\xi$-symmetric,
then for $\lam=\lam^*$ at least one of the following two cases occurs:
\begin{enumerate}[(i)]
\item $G_\lam$ is internally tangent to $\pa G$ at some point $P\in \pa G$
not in $\pi_\lam,$ or
\item $\pi_\lam$ is orthogonal to $\pa G$ at some point $Q$.
\end{enumerate}
\par
Now, the crucial remark is given by the following lemma, whose proof is
an easy adaptation of those of \cite{MS2}[Lemmas 2.1 and 2.2].

\begin{lemma}
\label{le:exterior sum}
 Let $G$ satisfy assumption \eqref{defD}. Then we have
\begin{itemize}
\item[(i)]
$\Om=G+B_\de(0)=\{ x+y: x\in G, y\in B_\de(0)\};$
\item[(ii)] if $\mathcal{R}_\lam(G_\lam)\subset G,$ then $\mathcal{R}_\lam(\Om_\lam)\subset \Om.$
\end{itemize}
\end{lemma}

Let $\Omega_\lam'$ denote the connected component of $\mathcal{R}_\lam(\Om_\lam)$
whose closure contains $P$ or $Q$.  We notice that, since $u$ is of class $C^1$ in a neighborhood of $\Gamma_\de$, Lemma \ref{lemma F2I gradient boundary} implies that that $|Du|$ is bounded away from $\sigma$ in the closure of a set $A_\de \supset \Gamma_\de.$
This information on $|Du|$ guarantees that Proposition \ref{thm weak comp minimiz} can be applied to
the two (local) minimizers $u$ and $u^\lam$ of $\II_{\Omega_\lam'}:$
since $u \geq u^\lam$ on $\pa \Omega_\lam'$ (and $|A_\de \cap \Omega_\lam|,$ $|A_\de \cap \Omega_\lam'|>0$),
then $u \geq u^\lam$ in $\Omega_\lam'$.

If case (i) occurs,
% the level surface $\Gamma_\de$ is tangent to its reflection $\mathcal{R}_\lam(\Gamma_\de)$ at some point $P_\de \not\in \pi_\lam$ which lies in the interior of $\Om_\lam'$.
we apply Proposition \ref{lemma F2I strong comparison} to $u$ and $u^\lam$ in $A_\de \cap \Omega_\lam'$
and obtain that $u>u^\lam$ in $A_\de \cap \Omega_\lam'$,
since $u \not\equiv u^\lam$ on $\Gamma_\de \cap \Omega_\lam'$. This is a contradiction, since $P$
belongs both to $A_\de \cap \Omega_\lam'$ and $\Gamma_\de \cap \mathcal{R}_\lam (\Gamma_\de),$ and hence $u(P) = u^\lam(P)$.

Now, let us consider case (ii). Notice that  $\xi$ belongs to the
tangent hyperplane to $\Gamma_\de$ at $Q$. Since $u\in C^1(A_\de)$ and $|Du|$ is bounded away from $\sigma$ in the closure of  $A_\de$, standard elliptic regularity theory (see \cite{To1} and \cite{GT}) implies that $u\in C^{2,\ga}(A_\de)$ for some $\ga \in (0,1)$.
Thus, applying Proposition \ref{lemma F2I boundary point}
to $u$ and $u^\lam$ in $A_\de \cap \Omega_\lam'$ yields
\begin{equation*}
\dfrac{\pa u}{\pa \xi}(Q) < \dfrac{\pa u^\lam}{\pa \xi}(Q).
\end{equation*}
On the other hand, since $\Gamma_\de$ is a level surface of $u$ and $u$ is differentiable at $Q$, we must have that
\begin{equation} \label{eq contr case ii}
\dfrac{\pa u}{\pa \xi}(Q) = \dfrac{\pa u^\lam}{\pa \xi}(Q) = 0;
\end{equation}
this gives the desired contradiction and concludes the proof of the theorem.

%\vspace{1em}

\begin{rmk}
Notice that, if we assume that $\si=0$, then the
assumption that $u$ is of class $C^1$ in a neighborhood of
$\Gamma_\de$ can be removed from Theorem \ref{thm symmetry diff
functionals}. Indeed, from elliptic regularity theory we have that $u\in C^{2,\ga}(\Omega \setminus \{Du =
0\})$, for some $\ga \in (0,1)$; from Lemma \ref{lemma F2I
gradient boundary} we know that $Du \neq 0$ on $\Gamma_\de$ and thus
$u\in C^{2,\ga}$ in an open neighborhood of $\Gamma_\de$. As far as we know, few regularity results are available in literature for the case $\si>0$ (see \cite{Br},\cite{BCS},\cite{CM} and \cite{SV}); since H\"{o}lder estimates for the gradient are missing, we have to assume that $u$ is continuously differentiable in a neighborhood of $\Gamma_\de$.
\end{rmk}

%\vspace{1em}

\begin{rmk}
We notice that \eqref{eq contr case ii} holds under the weaker assumption that $u$ is Lipschitz continuous as
we readily show.

Let $\xi$ and $Q$ be as in the proof of Theorem \ref{thm symmetry diff functionals}. For $\ep > 0$ small enough, we denote by $y(Q - \ep \xi)$ the projection of $Q - \ep \xi$ on $\Gamma_\de$. Since $\Gamma_\de$ is a level surface of $u$,
\begin{equation*}
u(Q-\ep \xi) - u(Q) =u(Q-\ep \xi) - u(y(Q-\ep \xi)),
\end{equation*}
and, being $u$ Lipschitz continuous, we have
\begin{equation*}
|u(Q-\ep \xi) - u(y(Q-\ep \xi))| \leq L |Q-\ep \xi - y(Q-\ep \xi)|,
\end{equation*}
for a positive constant $L$ independent of $\ep,\: \xi$ and $Q$. Since $\xi$ is a vector belonging to the tangent hyperplane to $\Gamma_\de$ at $Q$, we have that
\begin{equation*}
|Q-\ep \xi - y(Q-\ep \xi)| = o (\ep),
\end{equation*}
as $\ep \to 0^+$ and thus
\begin{equation*}
\lim_{\ep \to 0^+} \frac{u(Q-\ep \xi) - u(Q)}{\ep} = 0.
\end{equation*}
A similar argument applied to $u^\lam$ yields
\begin{equation*}
\lim_{\ep \to 0^+}  \frac{u^\lam(Q-\ep \xi) - u^\lam(Q)}{\ep} = 0,
\end{equation*}
and hence \eqref{eq contr case ii} holds.
\end{rmk}

\section{A class of non-differentiable functionals}\label{section minima not diff}
In this section we consider the variational problem \eqref{I(u)} and assume that $f$ satisfies (f1) and

\begin{itemize}
\item[(f3)] $f'(0) >0$, $f\in C^{2,\al}(0,+\infty)$ and $f''(s)>0$ for every $s>0$, with $0<\al<1$. \label{F2.II}
\end{itemize}

In this case, the function $s \to f(|s|)$ is not differentiable at
the origin and a minimizer of \eqref{I(u)} satisfies a
variational inequality instead of an Euler-Lagrange equation.
\par
By this non-differentiability of $\II_\Om,$
it may happen that $u \equiv 0$ is the minimizer of \eqref{I(u)}
when $\Om$ is ``too small'' (see Theorem \ref{lemma cheeger});
thus, it is clear that the symmetry result of Theorem \ref{thm symmetry diff
functionals} does not hold in the stated terms.
In Theorem \ref{thm symmetry not-diff functionals}, we will state
the additional conditions that enable us to extend Theorem \ref{thm symmetry diff functionals}
to this case.
\par
We begin with a characterization of solutions to \eqref{I(u)}.

\begin{proposition} \label{prop F2II diseq EL u}
Let $f$ satisfy (f1) and (f3). Then we have:
\begin{itemize}
\item[(i)] \eqref{I(u)} has a
unique solution $u;$

\item[(ii)] $u$ is characterized by the boundary condition, $u=0$
on $\pa \Omega,$ and the following inequality:
\begin{equation}\label{eqF2II.diseq EL u}
\Big{|} \ints_{\Omega^\sharp} f'(|Du|) \frac{Du}{|Du|} \cdot D\phi
\, dx -\ints_\Omega \phi dx \Big{|} \leq f'(0) \ints_{\Omega^0}
|D\phi| dx,
\end{equation}
for any $\phi \in C_0^1(\Omega).$ Here,
\begin{equation}\label{Omega0 e Omegadiesis}
\Omega^0 = \{ x\in\Om: Du(x) = 0\} \ \ \textmd{ and }\ \ \Omega^\sharp= \Omega \setminus \Omega^0.
\end{equation}
\end{itemize}
\end{proposition}

\begin{proof}
Since $f$ is strictly convex, the uniqueness of a minimizer follows
easily.

Let us assume that $u$ is a minimizer and let $\phi \in
C_0^1(\Omega)$; then
\begin{equation*}
\ints_{\Omega^\sharp} \frac{ f(|Du + \ep D\phi|) - f(|Du|)}{\ep} dx
+ \ints_{\Omega^0} \frac{f(\ep |D\phi|) - f(0)}{\ep} dx -
\ints_\Omega \phi dx  \geq 0,
\end{equation*}
for $\ep>0$. By taking the limit as $\ep \to 0^+,$ we obtain one of the two
inequalities in \eqref{eqF2II.diseq EL u}. The remaining inequality is obtained
by repeating the argument with $-\phi.$

Viceversa, let us assume that $u$ satisfies \eqref{eqF2II.diseq EL u} and let $\phi \in C_0^1(\Omega)$; the convexity of the function
$t\mapsto f(|Du + t D\phi|)$ yields
\begin{equation*}
\begin{split} \II_\Om(u+\phi)-\II_\Om(u) & = \ints_{\Omega^\sharp} [f(|Du + D\phi|) - f(|Du|)]\, dx +
\ints_{\Omega^0} f(|D\phi|)\, dx - \ints_\Omega \phi\, dx \\
& \geq \ints_{\Omega^\sharp} f'(|Du|) \frac{Du}{|Du|} \cdot D\phi \,
dx +  f'(0) \ints_{\Omega^0} |D\phi| -\ints_\Omega \phi\, dx \geq 0,
\end{split}
\end{equation*}
where the last inequality follows from \eqref{eqF2II.diseq EL u};
thus, $u$ is a minimizer of \eqref{I(u)}.
\end{proof}

%\vspace{1em}

Next, we recall the definition of the Cheeger constant $h(\Omega)$ of a set $\Omega$ (see \cite{Ch} and \cite{KF}):
\begin{equation}\label{h cheeger set}
h(\Omega)= \inf \left\{\frac{| \pa A |}{|A|}: A\subset\Om, \pa A\cap\pa\Om=\varnothing\right\}.
\end{equation}
It is well-known (see \cite{De} and \cite{KF}) that an equivalent definition of $h(\Omega)$ is given by
\begin{equation}\label{h cheeger}
h(\Omega) = \inf_{\phi \in C_0^1(\Omega)} \dfrac{\int_\Omega |D\phi|\,
dx}{\ints_\Omega |\phi|\, dx}.
\end{equation}

\begin{theorem} \label{lemma cheeger}
Let $u$ be the solution of \eqref{I(u)}, with $f$ satisfying (f1)
and (f3).
\par
Then $u\equiv 0 $ if and only if
\begin{equation} \label{fprime h geq 1}
f'(0)h(\Omega) \geq 1.
\end{equation}
\end{theorem}

\begin{proof}
We first observe that
\begin{equation} \label{h cheeger equiv}
h(\Omega) = \inf_{\phi \in C_0^1(\Omega)} \dfrac{\int_\Omega |D\phi|\,
dx}{|\ints_\Omega \phi\, dx|},
\end{equation}
since we can always assume that the minimizing sequences in \eqref{h cheeger} are made of non-negative functions.

Let us assume that $u=0$ is solution of \eqref{I(u)}, then from
\eqref{eqF2II.diseq EL u} and \eqref{h cheeger equiv} we easily get \eqref{fprime h geq 1}.
Viceversa, if \eqref{fprime h geq 1} holds, thanks to \eqref{h cheeger equiv}, $u \equiv
0$ satisfies \eqref{eqF2II.diseq EL u} and Proposition \ref{prop
F2II diseq EL u} implies that $u$ is the solution of \eqref{I(u)}.
\end{proof}

Observe that, if $\Om$ is a ball of radius $R,$ then its Cheeger constant is
\begin{equation} \label{cheeger costant for B R}
h(\Omega) = \frac{N}{R},
\end{equation}
as seen in \cite{KF}. Thus, Theorem \ref{lemma cheeger} informs us that $u\equiv 0$ is
the only minimizer of $\II_\Om$ if and only if $R\le N\,f'(0),$ i.e. if $\Om$ is small enough.
In the following proposition we get the explicit expression of
the solution of \eqref{I(u)} in a ball. Notice that in this case
the set $\Omega^0$ has always positive Lebesgue measure.

\begin{proposition} \label{prop F2II ball}
Let $f$ satisfy (f1) and (f3) and denote by $g$ the Fenchel conjugate of $f$.
Let $\Om\subset\RR^N$ be the ball of radius $R$
centered at the origin and let $u_R$ be the solution of
\eqref{I(u)}.
\par
Then $u_R$ is given by
\begin{equation} \label{eq soluzione palla F2II}
u_R(x) = \ints_{|x|}^R g'\Big( \frac{s}{N} \Big)\, ds , \quad 0\le |x|\le R.
\end{equation}
\end{proposition}

\begin{proof}
Under very general assumptions on $f$, a proof of this proposition can be found in \cite{Cr2}. In the following, we present a simpler proof which is {\it ad hoc} for the case we are considering.

As we have just noticed, if $R \leq Nf'(0),$ then  $h(B_R)\,f'(0) \geq 1$ and hence
Theorem \ref{lemma cheeger} implies that the minimizer of $\II_\Om$ must vanish everywhere.
Thus, \eqref{eq soluzione palla F2II} holds, since we know that $g'=0$ in
the interval $[0, f'(0)]$ and hence in $[0, R/N].$

Now, suppose that $R > N f'(0)$ and let $\phi \in C_0^1(B_R)$.
We compute the number between the bars in \eqref{eqF2II.diseq EL u} with $u=u_R;$
we obtain that
\begin{eqnarray*}
&&\ints_{\Omega^\sharp} f'(|Du|) \frac{Du}{|Du|} \cdot D\phi
\, dx -\ints_\Omega \phi\, dx=\\
&&-\!\!\!\!\!\!\!\!\ints_{N f'(0)<|x|<R} \frac{x}{N} \cdot D\phi \, dx - \ints_{|x|<R} \phi\, dx  = \ints_{|x|<Nf'(0)} \frac{x}{N} \cdot D\phi \, dx,
\end{eqnarray*}
after an application of the divergence theorem.
Applying the Cauchy-Schwarz inequality to the last integrand, we obtain that
$$
\Big{|} \ints_{\Omega^\sharp} f'(|Du|) \frac{Du}{|Du|} \cdot D\phi
\, dx -\ints_\Omega \phi\, dx \Big{|} \leq f'(0) \ints_{|x|<Nf'(0)} |D\phi|\,dx,
$$
that is \eqref{eqF2II.diseq EL u} holds; the conclusion then follows from Proposition \ref{prop F2II diseq EL u}.
\end{proof}

In the following two lemmas we derive the weak comparison principle and
Hopf lemma that are necessary to prove our symmetry result.

\begin{lemma}[Weak comparison principle] \label{lemma weak comp princ notdiff}
Let $f$ satisfy (f1) and (f3) and let $A$ be a bounded domain. Assume that $u_0, u_1 \in W^{1,\infty}(A)$ are minimizers of $\II_A$
such that $u_0 \leq u_1$ on $\pa A$.
\par
Then $u_0 \leq u_1$ on $\overline{A}$.
\end{lemma}

\begin{proof}
Let $B=\{x\in A: u_0(x)>u_1(x)\}$ and assume by contradiction
that $B \neq \varnothing$. Since $u_0 \leq u_1$ and being $u_0$ and $u_1$ both continuous, then $u_0=u_1$ on $\partial B$. Hence, $u_0$
and $u_1$ are two distinct solutions of the problem
\begin{equation*}
  \inf\bigg\{\int_B [f(|Du|)-u] dx : \  \ u(x)=u_0(x) \ \textmd{ on } \partial B \bigg\},
\end{equation*}
which is a contradiction, on account of the uniqueness of the minimizer of $\II_B$.
\end{proof}

\begin{lemma} \label{lemma F2II gradient boundary}
Let $f$ satisfy (f1) and (f3) and denote by $g$ the Fenchel conjugate of $f$. Let $\Omega \subset \RR^N$ be a bounded domain and let $G$ satisfy \eqref{defD}.

Assume that $u\in W_0^{1,\infty}(\Omega)$ is a minimizer of \eqref{I(u)}
satisfying \eqref{OVDTcond}. For $x_0\in \Gamma_\de,$ let $y_0\in \pa
\Omega$ be such that $\dist(y_0,\Gamma_\de)=\de$ and set $\nu = \frac{x_0-y_0}{\de}$; denote by $\rho=\rho(x_0)$ the radius of the optimal interior ball at $x_0$.
\par
Then,
\begin{equation*}
\liminf_{t\to 0^+} \frac{u(x_0+t \nu) - u(x_0)}{t} \geq g'\Big(
\frac{\rho}{N} \Big).
\end{equation*}

In particular, we have that $\ \inf\limits_{\Gamma_\delta} |Du| > 0$ in two cases:
\begin{itemize}
\item[(i)] if $u\in C^1(\Gamma_\delta)$ and $\rho(x_0) > Nf'(0)$ for any $x_0\in \Gamma_\delta$;
\item[(ii)] if $u$ is differentiable at every $x\in\Gamma_\de$ and $G$ satisfies the uniform interior sphere condition of radius $\rho>Nf'(0)$.
\end{itemize}
\end{lemma}

\begin{proof}
Notice that, if $\rho > Nf'(0)$ then Proposition \ref{prop F2II ball} implies that $u_R$ given by
\eqref{eq soluzione palla F2II} with $R=\rho$ is strictly positive in a ball of radius $\rho$.
Then, the proof of this lemma can be easily adapted from the proof of Lemma \ref{lemma F2I gradient boundary}.
\end{proof}

Finally, by repeating the argument of the proof of Theorem \ref{thm symmetry diff functionals}, we have the following theorem.

\begin{theorem} \label{thm symmetry not-diff functionals}
Let $f,\,\Omega$ and $G$ satisfy the assumptions of Lemma \ref{lemma F2II gradient boundary}. Assume that $u\in W_0^{1,\infty}(\Omega)$ is the minimizer of \eqref{I(u)} and that \eqref{OVDTcond} holds.
\par
If $u$ is of class $C^1$ in a tubular neighborhood of $\Gamma_\de$ and $\rho > Nf'(0),$ then $\Omega$ must be a ball.
\end{theorem}

\begin{proof}
The proof follows the lines of the proof of Theorem \ref{thm symmetry diff functionals}.

In this case, the weak comparison principle (which has to be applied to $u$ and $u^\lam$ in $\Omega_\lam'$) is given by Lemma \ref{lemma weak comp princ notdiff}.

Since $u$ is of class $C^1$ in an open neighborhood of $\Gamma_\de$, Lemma \ref{lemma F2II gradient boundary} implies that $|Du|$ is bounded away from zero in an open set $A_\de \supset \Gamma_\de$. Proposition \ref{prop F2II diseq EL u}
then implies that $u$ is a weak solution of
\begin{equation*}
-\diver\left\{f'(|Du|) \frac{Du}{|Du|}\right\} = 1
\end{equation*}
in $A_\de.$ Thus, a strong comparison principle and a Hopf comparison principle analogous to Propositions \ref{lemma F2I strong comparison} and \ref{lemma F2I boundary point} apply.

Once these three principles are established, the proof can be completed by using the method of moving planes, as done in Subsection \ref{subsection proof of thm}.
\end{proof}

\section{Symmetry results for fully nonlinear elliptic and parabolic equations}
\label{section4}

As already mentioned in the Introduction, the argument used in the proof of Theorem \ref{thm symmetry diff functionals} applies to more general elliptic equations of the form \eqref{elliptic}.
Following \cite{Se2} (see properties (A)-(D) on pp. 309-310), we chose to
state our assumptions on $F$ in a very general form and to refer the reader to
the vast literature for the relevant sufficient conditions.

Let $u$ be a viscosity solution of \eqref{elliptic} in $\Omega$ and assume that $u=0$ on $\pa \Omega$. Let $A\subseteq \Omega$ denote an open connected set.
\begin{itemize}
\item[(WCP)] We say that \eqref{elliptic} enjoys the \emph{Weak Comparison Principle} in $A$ if,
for any two viscosity solutions $u$ and $v$ of \eqref{elliptic}, the inequality $u \leq v$ on $\pa A$
extends to the inequality $u \leq v$ on $\overline{A}.$
\item[(SCP)] We say that \eqref{elliptic} enjoys the \emph{Strong Comparison Principle} in $A$ if
for any two viscosity solutions $u$ and $v$ of \eqref{elliptic}, the inequality $u \leq v$ on $\pa A$
implies that either $ u\equiv v$ in $\ovr{A}$ or $u < v$ in $A.$
\item[(BPP)] Suppose $\pa A$ contains a (relatively open) flat portion $H$.
We say that \eqref{elliptic} enjoys the \emph{Boundary Point Property} at $P\in H$ if, for any two solutions $u$ and $v$ of
\eqref{elliptic}, Lipschitz continuous in $A$ and such that $u \leq
v$ in $A,$ then the assumption $u(P)=v(P)$ implies that either $u
\equiv v$ in $\ovr{A}$ or else $u<v$ in $A$ and
\begin{equation*}
\limsup_{\ep\to 0^+}\frac{[v-u](P+\ep\nu)-[v-u](P)}{\ep}>0.
% \liminf_{\ep\to 0^+}\frac{u(P+\ep\nu)-u(P)}{\ep} < \liminf_{\ep\to 0^+}\frac{v(P+\ep\nu)-v(P)}{\ep}.
\end{equation*}
Here, $\nu$ denotes the inward unit normal to $\pa A$ at $P$.
\end{itemize}
\par
We shall also suppose that
\begin{itemize}
\item[(IR)]
equation \eqref{elliptic} is \emph{invariant
under reflections in any hyperplane;}
\end{itemize}
in other words, we require the following: for any $\xi$ and $\lam$, $u$ is a solution of \eqref{elliptic} in $\Omega$ if and only if $u^\lam$ is a solution of \eqref{elliptic} in $\RRlxi(\Omega)$ (here, we used the notations introduced in Subsection \ref{subsection proof of thm}).

The following symmetry results hold.
Here, we chose to state our theorems for {\it continuous} viscosity solutions;
however, the same arguments may be applied when the definitions of classical or weak solutions are considered.

\begin{theorem}
\label{th:symmEll}
Let $\Om\subset\RN$ be a bounded domain and let $G$ satisfy \eqref{defD}.
Let $u=u(x)$ be a non-negative viscosity solution of \eqref{elliptic}
satisfying the homogeneous Dirichlet boundary condition
\begin{equation*}
u=0 \ \mbox{ on } \ \pa\Om.
\end{equation*}
Suppose there exist constants $c>0$ and $\de>0$ such that \eqref{OVDTcond} holds.

Let $F$ satisfy (IR) and
\begin{itemize}
\item[(i)] (WCP) for $A=\Omega$;
\item[(ii)] (SCP) and (BBP) for some neighborhood $A_\de$ of $\Gamma_\de$.
\end{itemize}
Then $\Om$ must be a ball.
\end{theorem}

The corresponding result for parabolic equations reads as follows.

\begin{theorem}
\label{th:symmPar}
Let $F,\Omega$ and $G$ satisfy the assumptions of Theorem \ref{th:symmEll}.
\par
Let $u=u(x,t)$ be a non-negative viscosity solution of
\begin{eqnarray}
&& u_t-F(u, Du,D^2 u)=0 \ \mbox{ in } \ \Om\times(0,T),\label{parabolic}\\
&& u=0\ \mbox{ on } \ \Om\times\{0\},\label{initial}\\
&& u=1 \ \mbox{ on } \ \pa\Om\times(0,T). \label{parabDirichlet}
\end{eqnarray}
\par
If there exist a time $t^*\in(0,T)$ and constants $c>0$ and $\de>0$ such that
\begin{equation}
\label{OVDTparabcond}
u=c \ \mbox{ on } \ \Ga_\de\times\{t^*\},
\end{equation}
then $\Om$ must be a ball.
\end{theorem}

In the literature, there is a large number of results ensuring that (WCP),(SCP) and (BBP) hold
provided sufficient structure conditions are assumed on $F.$ In the following, we collect just few of them.

For a \emph{proper} equation (see \cite{CIL} for a definition) of the form \eqref{elliptic},
% \begin{equation*}
% F(u,Du,D^2u)=0,
% \end{equation*}
a weak comparison principle is given in \cite{KK2} (see also \cite{KK1} and \cite{BM}), where $F$ is assumed to be \emph{locally strictly elliptic} and to be \emph{locally Lipschitz continuous} in the second variable (the one corresponding to $Du$). Under the additional assumption that $F$ is \emph{uniformly elliptic}, (SCP) and (BBP) are proved in \cite{Tr}.
The assumptions in \cite{KK2} include some kinds of mean curvature type equations and nonhomogeneous $p$-Laplace equations; however they do not include the homogeneous $p$-Laplace equation and other degenerate elliptic equations.
\par
For nonlinear elliptic operators not depending on $u$, we can also consider some degenerate cases. Thus, for equations of the form
\begin{equation*}
F(Du,D^2u)=0,\quad  \textmd{in } \Omega,
\end{equation*}
a (WCP) can be found in \cite{BB} and a (SCP) and (BBP) is proved in \cite{GO}. Such criteria
include the $p$-Laplace equation and the minimal surface equation.

We conclude this section by mentioning that, for classical or distributional solutions, the reader can refer to the monographs \cite{PW},\cite{GT},\cite{Fr} and \cite{PS}. More recent and interesting developments on comparison principles for classical and viscosity solutions can be found in \cite{CLN1,CLN2,CLN3,DS,SS,Si}.


\begin{thebibliography}{CLN3}

\bibitem[BM]{BM} M. Bardi and P. Mannucci, {\em On the Dirichlet problem for non-totally degenerate fully
nonlinear elliptic equations}, Commun. Pure Appl. Anal., 5 (2006), pp. 709-–731.

\bibitem[BB]{BB} G. Barles and J. Busca, {\em Existence and Comparison Results for Fully Nonlinear Degenerate Elliptic Equations without Zeroth-Order Term}, Comm. Partial Diff. Equations, 26 (2001), pp. 2323--2337.

\bibitem[Br]{Br} L. Brasco, {\em Global $L^{\infty}$ gradient estimates for solutions to a certain degenerate elliptic equation}, Nonlinear Anal., 74 (2011), pp. 516–-531.

\bibitem[BCS]{BCS} L. Brasco, G. Carlier and F. Santambrogio, {\em Congested traffic dynamics, weak flows and very degenerate elliptic equations}, J. Math. Pures Appl., 93 (2010), pp. 163–-182.

\bibitem[CLN1]{CLN1} L. Caffarelli, Y.Y. Li and L. Nirenberg, {\em Some remarks on singular solutions of nonlinear elliptic equations. I}, Journal of Fixed Point Theory and Applications 5 (2009), 353-395.

\bibitem[CLN2]{CLN2} L. Caffarelli, Y.Y. Li and L. Nirenberg, {\em Some remarks on singular solutions of nonlinear elliptic equations. II: symmetry and monotonicity via moving planes}, preprint.

\bibitem[CLN3]{CLN3} L. Caffarelli, Y.Y. Li and L. Nirenberg, {\em Some remarks on singular solutions of nonlinear elliptic equations. III: viscosity solutions, including parabolic operators}, preprint.

\bibitem[CM]{CM} C. Carstensen and S. M\"{u}ller, {\em Local stress regularity in scalar nonconvex variational problems}, SIAM J. Math. Anal., 34 (2002), pp. 495-–509.

\bibitem[CeM]{CeM} S. Cecchini and R. Magnanini, {\em Critical points of solutions of degenerate elliptic equations in the plane}, Calc. Var. 39 (2010), pp. 121-–138.

\bibitem[Ch]{Ch} J. Cheeger, {\em A lower bound for the smallest eigenvalue of the Laplacian}, Problems in Analysis (Papers dedicated to S. Bochner), pp. 195--199.
    
\bibitem[CMS]{CMS} G. Ciraolo, R. Magnanini, and S. Sakaguchi, {\em Solutions of elliptic equations with a level surface parallel to the boundary: stability of the radial configuration}. Preprint 2013. ArXiv:1307.1257.


\bibitem[CIL]{CIL} M.G. Crandall, H. Ishii and P.L. Lions, {\em User's guide to viscosity solutions of second order partial differential equations}, Bull. Amer. Math. Soc. 27 (1992), pp. 1--67.

\bibitem[Cr1]{Cr1} G. Crasta, {\em A symmetry problem in the Calculus of Variations},
J. Eur. Math. Soc. (JEMS) 8 (2006), pp. 139--154.

\bibitem[Cr2]{Cr2} G. Crasta, {\em Existence, uniqueness and qualitative properties of  minima to
              radially symmetric non-coercive non-convex variational
              problems}, Math. Z., 235 (2000), pp. 569-589.

\bibitem[CG]{CG} G. Crasta and F. Gazzola, {\em Some estimates of the minimizing properties of web functions}, Calculus of Variations 15 (2002), pp. 45--66.

\bibitem[DS]{DS} F. Da Lio and B. Sirakov, {\em Symmetry results for viscosity solutions of fully nonlinear uniformly elliptic equations}, J. Eur. Math. Soc. 9 (2007), pp. 317--330.

\bibitem[De]{De} F. Demengel, {\em Some existence's results for noncoercive ``1-Laplacian'' operator}, Asymptot. Anal., 43 (2005), pp. 287–-322.

\bibitem[Fr]{Fr} L. E. Fraenkel, An Introduction to Maximum Principles and Symmetry in Elliptic Problems, Cambridge University Press, Cambridge, 2000.

\bibitem[FGK]{FGK} I. Fragal\`{a}, F. Gazzola and B. Kawohl, {\em Overdetermined problems with possibly degenerate ellipticity, a geometric approach},  Math. Z., 254 (2006), pp. 117--132.

\bibitem[Ga]{Ga} F. Gazzola, {\em Existence of minima for nonconvex functionals in spaces of functions depending on the distance from the boundary}, Arch. Rat. Mech. Anal. 150 (1999), pp. 57--76.

\bibitem[GNN]{GNN} B. Gidas, W-M. Ni and L. Nirenberg, {\em Symmetry and related properties via the maximum principle}, Comm. Math. Phys., 68 (1979), pp. 209-–243.

\bibitem[GO]{GO} Y. Giga and M. Ohnuma, {\em On strong comparison principle for semicontinuous viscosity solutions of some nonlinear elliptic equations}, Int. J. Pure Appl. Math. 22 (2005), pp. 165–-184.

\bibitem[GT]{GT} D. Gilbarg and N.S. Trudinger, Elliptic partial differential equations of second order, Springer-Verlag, Berlin-New York, 1977.

\bibitem[Je]{Je} R. Jensen, {\em The maximum principle for viscosity solutions of fully nonlinear second order partial differential equations},  Arch. Rational Mech. Anal., 101 (1988), pp. 1--27.

\bibitem[KF]{KF} B. Kawohl, V. Fridman, {\em Isoperimetric estimates for the first eigenvalue of the $p$-Laplace operator and the Cheeger constant}, Comment. Math. Univ. Carol., 44 (2003), pp. 659--667.

\bibitem[KK1]{KK1} B. Kawohl, N. Kutev, {\em Comparison principle and Lipschitz regularity for viscosity solutions of some classes of nonlinear partial differential equations}, Funkcial. Ekvac. 43 (2000), pp. 241--253.

\bibitem[KK2]{KK2} B. Kawohl, N. Kutev, {\em Comparison Principle for Viscosity Solutions of Fully Nonlinear, Degenerate Elliptic Equations}, Comm. Partial Differential Equations, 32 (2007), no. 8, pp. 1209--1224.

\bibitem[MS1]{MS} R. Magnanini and S. Sakaguchi, {\em Nonlinear diffusion with a bounded stationary level surface}, Ann. Inst. H. Poincaré Anal. Non Lin{\'e}aire, 27 (2010), pp. 937–-952.

\bibitem[MS2]{MS2} R. Magnanini and S. Sakaguchi, {\em Matzoh ball soup revisited:
the boundary regularity issue}, to appear in Math. Meth. Appl. Sci.

\bibitem[MT1]{MT1} R. Magnanini and G. Talenti, {\em On complex-valued solutions to a 2D eikonal equation. Part one: qualitative properties}, Nonlinear Partial Differential Equations, Contemporary Mathematics 283 (1999), American Mathematical Society, pp. 203--229.

\bibitem[MT2]{MT2} R. Magnanini and G. Talenti, {\em On complex-valued solutions to a 2D eikonal equation. Part two: existence theorems}, SIAM J. Math. Anal. 34 (2003), pp. 805--835.

\bibitem[MT3]{MT3} R. Magnanini and G. Talenti, {\em On complex-valued solutions to a 2D Eikonal Equation. Part Three: analysis of a Backlund transformation}, Applic. Anal. 85 (2006), no. 1-3, pp. 249--276.

\bibitem[MT4]{MT4} R. Magnanini and G. Talenti, {\em On complex-valued 2D eikonals. Part four: continuation past a caustic}, Milan Journal of Mathematics 77 (2009), no. 1, pp. 1--66.

\bibitem[PW]{PW} M.H. Protter and H.F. Weinberger, Maximum principles in differential equations, Springer-Verlag, New York, 1984.

\bibitem[PS]{PS} P. Pucci and J. Serrin, The maximum principle, Birkh\"{a}user Verlag, Basel, 2007.

\bibitem[SV]{SV} F. Santambrogio and V. Vespri, {\em Continuity in two dimensions for a very degenerate elliptic equation}, Nonlinear Anal., 73 (2010), pp. 3832-–3841.

\bibitem[Se1]{Se1} J. Serrin, {\em On the strong maximum principle for quasilinear second order differential inequalities}, J. Functional Analysis, 5 (1970), pp. 184--193.

\bibitem[Se2]{Se2} J. Serrin, {\em A symmetry problem in potential theory}, Arch. Rational Mech. Anal., 43 (1971), pp. 304--318.

\bibitem[Sh]{Sh} H. Shahgholian, {\em Diversifications of Serrin's and related symmetry problems}, in press in Complex variables and elliptic equations.

\bibitem[SS]{SS} L. Silvestre and B. Sirakov, {\em Overdetermined problems for fully nonlinear elliptic equations}, preprint (2013). ArXiv:1306.6673.

\bibitem[Si]{Si} B. Sirakov, {\em Symmetry for exterior elliptic problems and two conjectures in potential theory}, Ann. Inst. Henri Poincar\'e, Anal. non lin\'{e}aire 18 (2001), pp. 135--156.

\bibitem[To1]{To1} P. Tolksdorf, {\em Regularity for a more general class of quasilinear elliptic equations}, J. Differential Equations, 51 (1984), pp. 126–-150.

\bibitem[To2]{To2} P. Tolksdorf, {\em On the Dirichlet problem for quasilinear equations in domains with conical boundary points}, Comm. Partial Differential Equations, 8 (1983), pp. 773-–817.

\bibitem[Tr]{Tr} N. Trudinger, {\em Comparison principles and pointwise estimates for viscosity solutions}, Rev. Mat. Iberoamericana, 4  (1988), pp. 453--468.


\end{thebibliography}
\end{document}